\documentclass[11pt]{amsart}
\usepackage{graphicx}
\usepackage{amssymb, amsmath}
\newtheorem{theorem}{Theorem}[section]
\newtheorem{corollary}[theorem]{Corollary}
\newtheorem{lemma}[theorem]{Lemma}

\theoremstyle{definition}

\theoremstyle{remark}

\begin{document}
\title{Existentially closed structures and some embedding theorems}
\author{\sc M. Shahryari}
\thanks{}
\address{ Department of Pure Mathematics,  Faculty of Mathematical
Sciences, University of Tabriz, Tabriz, Iran }

\email{mshahryari@tabrizu.ac.ir}
\date{\today}

\begin{abstract}
Using the notion of existentially closed structures, we obtain
embedding theorems for groups and Lie algebras. We also prove the
existence of some groups and Lie algebras with prescribed
properties.
\end{abstract}

\maketitle


\vspace{2cm}

An algebraic system $A$ is called {\em existentially closed}, if
every consistent finite set of existential sentences with parameters
from $A$, is satisfiable in $A$. Using the disjunctive normal form
of existential sentences, one can easily see that $A$ is
existentially closed, if and only if any consistent finite set of
equations and in-equations with parameters from $A$, is satisfiable
in $A$. As a special case,  we can define the notion of
existentially closed groups. It can be shown that (see \cite{H-S}),
a group $G\neq 1$ is  existentially closed group, if and only if
every consistent finite system of equations with parameters from
$G$, has a solution in $G$. In the general case, an algebraic system
with the property that every consistent finite set of equations has
a solution in that algebraic system is called {\em algebraically
closed}. So, in the case of groups (and Lie algebras), two
properties of being existentially closed and being algebraically
closed are equivalent.

The notion of existentially closed groups is introduced by a short
paper of W. R. Scott in 1951, and he used the axiom of choice to
prove the existence of such groups, \cite{Scott}. Since then, the
class of existentially closed groups is studied extensively and many
interesting properties of such groups are discovered during the past
decades. For a review of the history of existentially closed groups,
the reader can see \cite{H-S} or \cite{Lynd}.

 In this article, using the concept of existentially closed groups and Lie algebras, we prove some embedding theorems.
For example, we will show that for any set of primes $\pi$, any
infinite T$_{\pi}$-group can be embedded in a simple
$\pi^{\prime}$-divisible T$_{\pi}$-group with the same cardinality.
Further in the larger group,  elements of the same orders are
conjugate. By a T$_{\pi}$-group, we mean a group $G$ such that the
equality $x^m=1$ implies $x=1$, whenever $m$ is a
$\pi^{\prime}$-number.  By the notion of existentially closed Lie
algebras, we show that any Lie algebra $L$ can be embedded in a
simple Lie algebra $L^{\ast}$ which has many interesting properties;
for example, for any non-zero elements $a$ and $b$, there is $x$
such that $[x,a]=b$. In the case of finite fields, every finite
dimensional Lie algebra can be embedded in $L^{\ast}$ and it is
possible to describe the derivation algebra of any finite
dimensional algebra $A$ as the quotient algebra
$N_{G^{\ast}}(A)/C_{G^{\ast}}(A)$. We also prove the existence of
some groups and Lie algebras with prescribed properties. Our main
tool in this work is HNN-extensions of groups and Lie algebras. The
first one is well-known and the reader can consult any book on
combinatorial group theory (for example \cite{Lynd}) to see
definition and properties of HNN-extensions of groups. The
HNN-extensions of Lie algebras are not so popular and it seems that
there are only two articles ever published in the subject,
\cite{shirvani} and \cite{Alon}.

Although, we are dealing just with groups and Lie algebras in this
article, it is useful to give a general definition of existentially
closed structures in the frame of the universal algebra. Let
$\mathcal{L}$ be an algebraic language and $A$ be an algebra of type
$\mathcal{L}$. We extend $\mathcal{L}$ to a new language
$\mathcal{L}_A$ by adding new constant symbols $c_a$ for any $a\in
A$. Let $T_A(X)$ be the term algebra of $\mathcal{L}_A$ with
variables from a countable set $X$. If $p(x_1, \ldots, x_n)$ and
$q(x_1, \ldots, x_n)$ are elements of this term algebra, then we
call the expression $p(x_1, \ldots, x_n)=q(x_1, \ldots, x_n)$ an
equation with coefficients (or parameters) from $A$. An in-equation
is the negation of an equation. A system of equations and
in-equations over $A$ (or a system over $A$) is a finite set
consisting equations and in-equations. We say that $A$ is
existentially closed, if and only if any system over $A$ having a
solution in an extension $B$ of $A$, has already a solution in $A$.
In this article, we consider the case of groups and Lie algebras,
but some of the theory here, can be generalized to arbitrary
algebraic structures. Note that in the case of a group $G$, we may
assume that an equation has the form $w(x_1, \ldots, x_n)=1$, where
$w$ is an element of the free product $G\ast F(X)$. Here $F(X)$ is
the free group on the set $X$. Similarly, if $L$ is a Lie algebra
(or any non-associative algebra), then an equation over $L$ has the
form $w(x_1, \ldots, x_n)=0$, with $w$ and element of the free
product $L\ast F(X)$, where $F(X)$ is the free Lie algebra over $X$.

\section{Existentially closed groups}

A group $G$ is called {\em existentially closed}, if any finite
consistent system of equations and in-equations with coefficients
from $G$ has a solution in $G$. A system
$$
S=\{ w_i(x_1, \ldots, x_n)=1; (1\leq i\leq r), w_j(x_1, \ldots,
x_n)\neq 1; (r+1\leq j\leq s)\}
$$
with coefficients in $G$ is called consistent, if there is a group
$K$ containing $G$, such that $S$ has a solution in $K$. One can
generalize this definition to an arbitrary class of groups: Let
$\mathfrak{X}$ be a class of groups. A group $G\in \mathfrak{X}$ is
called existentially closed in the class $\mathfrak{X}$, if every
$\mathfrak{X}$-consistent system $S$ has a solution in $G$. Here,
$\mathfrak{X}$-consistency means that there exists a group $K\in
\mathfrak{X}$ which contains $G$ and $S$ has a solution in $K$. As
we said in the introduction, if $\mathfrak{X}$ is the class of all
groups, then existentially closedness is the same as algebraically
closedness.

The next lemma and theorem are proved in \cite{Scott} for the class of all groups, but we give here the proofs again for the sake of completeness.
Recall that a class of groups is called {\em inductive}, if it contains the union of any chain of its elements. \\

\begin{lemma}
Let $\mathfrak{X}$ be an inductive class of groups which is closed
under the operation of taking subgroups. Let $G\in \mathfrak{X}$.
Then there is a group $H\in \mathfrak{X}$ which contains $G$ and its
order is at most $\max \{ \aleph_0, |G|\}$. Further, for any finite
system $S$ of equations and in-equations over $G$, we have either of
the following assertions:

1- $S$ has a solution in $H$.

2- For any extension $H\subseteq E\in \mathfrak{X}$, the system $S$ has no solution in $E$.

\end{lemma}

\begin{proof}
Suppose $\kappa=\max \{ \aleph_0, |G|\}$. Clearly the cardinality of
the set of all finite systems of equations and inequations over $G$,
is also at most $\kappa$. We suppose that this set is well-ordered
as $\{ S_{\alpha}\}_{\alpha}$, by the ordinals $0\leq \alpha \leq
\kappa$. Let $G_0=G$ and suppose that $G_{\gamma}\in \mathfrak{X}$
is already defined in such a way that $|G_{\gamma}|\leq \kappa$ and
$\beta<\gamma$ implies $G_{\beta}\subseteq G_{\gamma}$. Let
$$
K_{\alpha}=\bigcup_{\gamma<\alpha}G_{\gamma}.
$$
Clearly, $K_{\alpha}\in \mathfrak{X}$ and $|K_{\alpha}|\leq \kappa^2=\kappa$. If $S_{\alpha}$ has no solution in any extension of $K_{\alpha}$,
then we set $G_{\alpha}=K_{\alpha}$, otherwise there is a $E\in \mathfrak{X}$ which contains $K_{\alpha}$ and $S_{\alpha}$ has a solution,
say $\bar{u}=(u_1, \ldots, u_n)$ in $E$ ($n$ is the number of indeterminate in $S_{\alpha}$). Let
$$
G_{\alpha}=\langle K_{\alpha}, u_1, \ldots, u_n\rangle \leq E.
$$
Then $G_{\alpha}\in \mathfrak{X}$ and $|G_{\alpha}|\leq \kappa$. So, for any $\alpha<\kappa$, we have defined a $G_{\alpha}$. Note that, we have also
$$
\beta<\alpha \Rightarrow G_{\beta}\subseteq G_{\alpha}.
$$
Now, the group $H=\cup G_{\alpha}\in \mathfrak{X}$ has the required properties. This is because, any system $S=S_{\alpha}$ which has no solution in $H$,
has no solution in $K_{\alpha}$ as well. So it has no solution in any extension of $H$.
\end{proof}

\begin{theorem}
Let $\mathfrak{X}$ be an inductive class of groups which is closed under the operation of taking subgroups. Let $G\in \mathfrak{X}$. Then,
there exists a group $G^{\ast}\in \mathfrak{X}$, with the following properties,

1- $G$ is a subgroup of $G^{\ast}$.

2- $G^{\ast}$ is existentially closed in the class $\mathfrak{X}$.

3- $|G^{\ast}|\leq \max \{ \aleph_0, |G|\}$.
\end{theorem}

\begin{proof}
Let $G^0=G$ and $G^1=H$, where $H$ is the group constructed in the lemma. Suppose $G^m$ is already defined and $G^{m+1}$ is the group which is
proved to does exist for $G^m$ in the lemma. Let $G^{\ast}=\cup G^m$. Therefore, $G^{\ast}\in \mathfrak{X}$, satisfies conditions 1 and 3.
To prove 2, suppose $S$ is a consistent system, with coefficients from $G^{\ast}$. Since $S$ is finite, so there is an $m$ such that all of the
coefficients of $S$ belong to $G^m$. So, $S$ has a solution in $G^{m+1}\subseteq G^{\ast}$.
\end{proof}

Note that in the number 3 of the above theorem, we may have
$|G^{\ast}|< \max \{ \aleph_0, |G|\}$. For example, suppose $G$ is
an arbitrary group and $\mathfrak{X}$ is the set of all subgroups of
$G$. Then trivially, $G^{\ast}=G$ and so, if $G$ is finite, then the
inequality 3 is strict. Note that also, there are many inductive
classes of groups, which are closed under subgroup: any variety,
quasi-variety and in general, any universal class of groups has this
property.  As the first application of the above theorem, we show
that every torsion free group can be embedded in a  torsion free
group with exactly two conjugacy classes.  Note that we can use a
similar arguments to prove the existence of torsion free groups of
any infinite cardinality with just two conjugacy classes. Also note
that the type of the group we are giving here, is not new, it is known from works of Higman, Neumann and Neumann \cite{Lynd}.\\

\begin{corollary}
Every torsion free group $G$ can be embedded in a torsion free group
$G^{\ast}$ which has just two conjugacy classes and its cardinality
is the same as the cardinality of $G$.
\end{corollary}

\begin{proof}
Suppose $\mathfrak{X}$ is the class of all torsion free groups.
Hence $\mathfrak{X}$ is inductive and closed under the operation of
taking subgroups. We, begin with the group $G\in \mathfrak{X}$.
Suppose $G^{\ast}\in \mathfrak{X}$ is the existentially closed group
relative to $\mathfrak{X}$, which is constructed for $G$ in the
theorem. We show that $G^{\ast}$ is the required group. Let $a, b\in
G^{\ast}$ be two non-identity elements. Consider the equation
$xax^{-1}=b$. Let
$$
G^{\ast}_{a, b}=\langle G^{\ast}, t:\ tat^{-1}=b\rangle
$$
be an HNN-extension of $G^{\ast}$. We know that every torsion
element of this HNN-extension is conjugate to a torsion element of
$G^{\ast}$, so $G^{\ast}_{a, b}$ is torsion free. It also contains
$G^{\ast}$ as a subgroup and clearly $t$ is a solution for
$xax^{-1}=b$ in $G^{\ast}_{a, b}$. Therefor, there is already a
solution in $G^{\ast}$. Hence $G^{\ast}$ is a  torsion free group,
containing $G$,  with just two conjugacy classes. Further
$|G^{\ast}|=|G|$.
\end{proof}

\section{An embedding theorem for T$_{\pi}$-groups}

It is easy to see that the simple group $G^{\ast}$ obtained in the
corollary 1.3 is divisible. In this section, we prove a similar
result for T$_{\pi}$-groups. Let $\pi$ be a set of primes and $G$ be
a group, such that its torsion elements have $\pi$-orders, i.e. the
equality $x^m=1$ implies that $x=1$ or $m$ is a $\pi$-number. Then
we say that $G$ is a T$_{\pi}$-group. If $\pi=\emptyset$, then
T$_{\pi}$-groups are exactly torsion free groups. Note that a finite
T$_{\pi}$-group is just a finite $\pi$-group. A group $G$ is also
said to be a $\pi^{\prime}$-group, if for any $\pi^{\prime}$-number
$m$, every element of $G$ has an $m$-th root.

\begin{theorem}
Let $G$ be a T$_{\pi}$-group. Then there exists a
$\pi^{\prime}$-divisible simple  T$_{\pi}$-group $G^{\ast}$
containing $G$ and with the cardinality  $\max \{ \aleph_0, |G|\}$,
such that

1- element of the same order in $G^{\ast}$ are conjugate,

2- $G^{\ast}$ is not finitely generated,

3-  every finite $\pi$-group embeds in $G^{\ast}$.

4- every finitely presented $\pi$-group can be residually embedded
in $G^{\ast}$,

5- for any finite $\pi$-group $A$, we have $Aut(A)\cong
\frac{N_{G^{\ast}}(A)}{C_{G^{\ast}}(A)}$.
\end{theorem}

\begin{proof}
Let $\mathfrak{X}_{\pi}$ be the class of all T$_{\pi}$-groups. This
class is  inductive and closed under subgroup. Since $G\in
\mathfrak{X}_{\pi}$, so there exists a group $G^{\ast}\in
\mathfrak{X}_{\pi}$ containing $G$, which is existentially closed in
the class $\mathfrak{X}_{\pi}$. Suppose $1\neq a, b\in G^{\ast}$
have the same orders. We show that $a$ and $b$ are conjugate.
Suppose the HNN-extension
$$
G^{\ast}_{a,b}=\langle G^{\ast}, t: tat^{-1}=b\rangle.
$$
It is enough to show that $G^{\ast}_{a,b}$ belongs to
$\mathfrak{X}_{\pi}$. Let $x$ be an element of finite order in
$G^{\ast}_{a, b}$. We know that $x$ is conjugate to an element of
$G^{\ast}$, see \cite{Lynd}. So the order of $x$ is a $\pi$-number,
showing that $G^{\ast}_{a,b}\in \mathfrak{X}_{\pi}$.  This shows
that elements of the same order in $G^{\ast}$ are conjugate. To show
that $G^{\ast}$ is $\pi^{\prime}$-divisible, let $x\in G^{\ast}$ and
$m$ be a $\pi^{\prime}$-number. We know that the orders of $x$ and
$x^m$ are equal. Hence there is a $z\in G^{\ast}$ such that
$x=zx^mz^{-1}$. Suppose $u=zxz^{-1}$. Then $u^m=x$ and hence
$G^{\ast}$ is $\pi^{\prime}$-divisible.

We show that $G^{\ast}$ is simple. Note that $G^{\ast}\ast \langle
x\rangle$ is a T$_{\pi}$-group. This follows from the fact that
$G^{\ast}$ is an T$_{\pi}$-group and the the reduced form of
elements in free products is unique. Let $1\neq a, w\in G^{\ast}$ be
arbitrary elements and suppose $u=wxw^{-1}x^{-1}$ and
$v=axw^{-1}x^{-1}$. Then $u$ and $v$ are reduced in the free product
and so they have infinite orders. Hence we can consider the
HNN-extension
$$
M=\langle G^{\ast}\ast \langle x\rangle, t: tut^{-1}=v\rangle.
$$
With the same argument (as for $G^{\ast}_{a,b}$), we see that $M$ is
also an T$_{\pi}$-group. The equation
$$
twxw^{-1}x^{-1}t^{-1}=axw^{-1}x^{-1}
$$
has a solution for $t$ and $x$ in $M$, therefore it has already a solution in $G^{\ast}$. We have
$$
a=(twt^{-1})(txw^{-1}x^{-1}t^{-1})(xw^{-1}x^{-1}),
$$
so $a\in \langle w^{G^{\ast}}\rangle$. Hence, for all $1\neq w\in G^{\ast}$, we have $G^{\ast}=\langle w^{G^{\ast}}\rangle$.
This proves that $G^{\ast}$ is simple.

Now, since $G^{\ast}$ is non-abelian simple group, we have $Z(G^{\ast})=1$. On the other hand, for any finite subset
$a_1, \ldots, a_m\in G^{\ast}$, the system
$$
a_1x=xa_1, \ldots, a_mx=xa_m, x\neq 1,
$$
has a solution in the T$_{\pi}$-group $G^{\ast}\times \langle
x\rangle$, so it has a solution in $G^{\ast}$. Therefore
$$
C_{G^{\ast}}(\langle a_1, \ldots, a_m\rangle)\neq 1.
$$
This proves that $G^{\ast}$ is not finitely generated.

Now, suppose that $A$ is a finite $\pi$-group and
$$
A=\{ 1=g_0, g_1, \ldots, g_m\}.
$$
For any $1\leq i, j\leq m$ there exists a unique $0\leq k(i, j)\leq
m$, such that $g_ig_j=g_{k(i,j)}$. The group $G^{\ast}\times A$ is
clearly an T$_{\pi}$-group and the system
$$
x_ix_j=x_{k(i.j)}; (0\leq i, j\leq m); x_i\neq x_j; (0\leq i<j\leq m)
$$
has a solution in $G^{\ast}\times A$ and hence it has a solution in
$G^{\ast}$. This shows that $A$ is embedded in $G^{\ast}$. To prove
$4$, let $A$ be a finitely presented $\pi$-group, with a finite
presentation
$$
A=\langle x_1, \ldots, x_n: r_1, \ldots, r_m\rangle.
$$
Suppose $1\neq w\in A$ and consider the system
$$
r_i(x_1, \ldots, x_n)=1; (1\leq i\leq m), w(x_1, \ldots,x_n)\neq 1.
$$
Clearly this equation has a solution in $G^{\ast}\times A\in
\mathfrak{X}_{\pi}$, and hence it has a solution in $G^{\ast}$. This
proves that there is a homomorphism $\varphi: A\to G^{\ast}$ with
$\varphi(w)\neq 1$. Hence $A$ embeds in $G^{\ast}$ residually. We
now prove $5$. Suppose again that $A$ is a finite $\pi$-group and
$\alpha\in Aut(A)$. This time, we consider the system
$$
xax^{-1}=\alpha(a);\ (a\in A).
$$
This system has a solution in the HNN-extension
$$
G^{\ast}_{\alpha}=\langle G^{\ast}, t: tat^{-1}=\alpha(a); (a\in A)\rangle.
$$
 Therefore, there exists an element $x\in G^{\ast}$ such that $\alpha$ is equal to the restriction of
the inner automorphism $T_x$ to $A$. This also shows that $x\in
N_{G^{\ast}}(A)$ and so the map which sends $x$ to the restriction
of $T_x$ on $A$ is an epimorphism from $N_{G^{\ast}}(A)$ to $Aut(A)$
with the kernel $C_{G^{\ast}}(A)$. This completes the proof of $5$.
\end{proof}

We can generalize the above theorem for HNN-classes of groups. A
class $\mathfrak{X}$ of groups is called an HNN-class if it is
inductive, closed under subgroup, and  HNN-extensions of the
elements of $\mathfrak{X}$ belong to $\mathfrak{X}$. The class of
all T$_{\pi}$-groups is an example of HNN-classes. By a similar
argument as in the above theorem, we can prove the next result.

\begin{theorem}
Let $\mathfrak{X}$ be an HNN-class of groups. Then every element $G$
of $\mathfrak{X}$ embeds in a simple $G^{\ast}\in \mathfrak{X}$,
with the cardinality at most $\max \{ \aleph_0, |G|\}$. Further
elements of the same orders in $G^{\ast}$ are conjugates.
\end{theorem}

\section{Some Olshanskii like groups}
In mid twenties, Alfred Tarski asked about the existence of
infinite groups all proper non-trivial subgroups of which are of
fixed prime order $p$. In 1982, A. Yu. Olshanskii \cite{Olshan},
constructed an uncountable family of such groups using his geometric
method of graded diagrams over groups, for all primes $p>10^{75}$.
The groups constructed are called Tarskii monsters since then. These
groups are two-generator simple groups and hence are countable. In
this section, for any fixed prime $p$,  we give a quite elementary
proof for existence of countable non-abelian simple groups with the
property that their all  non-trivial {\em finite subgroups} are
cyclic of order $p$.

We will consider two special classes of groups in this section. The
first one consists of  groups all finite  subgroups in which are
cyclic. We will denote this class by $\mathfrak{X}_{fc}$. The second
class which will be denoted by $\mathfrak{X}_p$, is the class of all
groups in which their non-trivial finite subgroups are of order $p$,
for a fixed prime $p$. Note that both classes are HNN-classes.
Clearly the Monsters constructed by Olshanskii satisfy the
requirements of the next theorem, but we don't use that monsters,
since we have a very elementary proof for our claims. What we  need
is the theorem 1.2 and some facts about finite subgroups of
HNN-extensions (and also those of free products). It is known that
(see \cite{Lynd}, page 212) every finite subgroup of any
HNN-extension
$$
G=\langle A, t: tFt^{-1}=\phi(F)\rangle
$$
is contained in some conjugate of $A$. Also, every finite subgroup
of any  free product $A\ast B$ is contained in some conjugate of $A$
or some conjugate of $B$.

\begin{theorem}
There exists a countable non-abelian simple group $M$ such that all
finite subgroups of $M$ are cyclic and for any prime $p$, the group
$M$ has an element of order $p$.
\end{theorem}

\begin{proof}
We know that the class $\mathfrak{X}_{fc}$ is inductive and closed
under subgroup, so we can apply 1.2. Note that, we have many groups
which belong to $\mathfrak{X}_{fc}$, for example any finite cyclic
group, $\mathbb{Z}_{p^{\infty}}$, and any torsion free group. So,
let $G\in \mathfrak{X}_{fc}$ be a countable arbitrary element.
Therefore, there exists a group $G^{\ast}$ satisfying requirements
of 1.2. Let $M=G^{\ast}$ and $p$ be a prime. We show that $M$ has an
element of order $p$. Suppose this is not true. Hence the system
$$
x^p=1, \ x\neq 1
$$
has no solutions in $M$. Let $H\leq M\times \mathbb{Z}_p$ be a
finite subgroup. We have $H\subseteq \pi_1(H)\times \pi_2(H)$, where
$\pi_1$ and $\pi_2$ are projections. Since $M\in \mathfrak{X}_{fc}$,
so $\pi_1(H)$ is cyclic and by our assumption $p$ is co-prime to the
order of $\pi_1(H)$. Therefore $\pi_1(H)\times \pi_2(H)$ is cyclic
and so is $H$. Hence $M\times \mathbb{Z}_p$ belongs to
$\mathfrak{X}_{fc}$. But the above system has a solution in $M\times
\mathbb{Z}_p$, a contradiction.

Note that, for any cyclic group $\langle x\rangle$ the group $M\ast
\langle x\rangle$ belongs to $\mathfrak{X}_{fc}$. To see this, let
$H$ be a finite subgroup of this free product. Then  $H$ must be
contained in a conjugate of $M$ and hence $H$ is cyclic. Now,
suppose $1\neq a, b\in M$ and consider two elements
$$
u=axa^{-1}x^{-1},\ v=bxa^{-1}x^{-1}
$$
in $M\ast \langle x\rangle$. Clearly $u$ and $v$ have infinite
orders and so we can consider the HNN-extension
$$
M^{\ast}=\langle M\ast \langle x\rangle, t:\ tut^{-1}=v\rangle.
$$
Let $H$ be a finite subgroup of this HNN-extension. Then $H$ is
contained in some conjugate of $M$ and therefore it is cyclic. Hence
$M^{\ast}$ is also an element of $\mathfrak{X}_{fc}$. So, the
equation
$$
taxa^{-1}x^{-1}t^{-1}=bxa^{-1}x^{-1}
$$
has a solution for $t$ and $x$ in $M^{\ast}$, therefore it has
already a solution in $M$. We have
$$
b=(tat^{-1})(txa^{-1}x^{-1}t^{-1})(xa^{-1}x^{-1}),
$$
so $b\in \langle a^{M}\rangle$. Hence, for all $1\neq a\in M$, we
have $M=\langle a^{M}\rangle$. This proves that $M$ is simple.

Finally note that since for any $a\in M$, the in-equation $ax\neq
xa$ has a solution in $M\ast\langle x\rangle$, so the center of $M$
is trivial and hence $M$ is non-abelian.
\end{proof}

Now, a similar argument on the class $\mathfrak{X}_p$ proves the
next theorem.

\begin{theorem}
Let $p$ be a fixed prime. Then there exists a countable non-abelian
simple group $M$ (which is not torsion free) such that any finite
non-trivial subgroup of $M$ is cyclic of order $p$.
\end{theorem}

The following two corollaries can be deduced instantly.

\begin{corollary}
There  exists a non-abelian  two  generator group such that the
orders of the generators are distinct primes and its every finite
subgroup  is cyclic.
\end{corollary}

\begin{corollary}
There exists a non-abelian $p$-group with two generators such that
its every finite subgroup has order $p$.
\end{corollary}

\section{Existentially closed algebras}
We are going now  to find similar embedding theorems for Lie
algebras. But there are two major differences between groups and Lie
algebras. First, we don't have any suitable definition of {\em
torsion} in the case of Lie algebras so, in advance, we don't have a
parallel concept of T$_{\pi}$-Lie algebra and so on. Instead, we can
express our theorems in terms of arbitrary Lie algebras. The second
main difference is related to HNN-extensions of Lie algebras. Here,
an HNN-extension comes from a Lie algebra and a derivation of some
subalgebra, despite  groups where HNN-extensions are always defined
by groups and isomorphisms  between subgroups. We will give a brief
summary of HNN-extensions of Lie algebras in the next section. In
this section, we give the analogue of Theorem 1.2 for Lie algebras,
in fact since it can be formulated  for arbitrary non-associative
algebras, we prove it in the most general form.

\begin{lemma}
Let $\mathfrak{X}$ be an inductive class of (not necessarily associative) algebras over a field $K$. Suppose $\mathfrak{X}$ is closed under
subalgebra and $L\in \mathfrak{X}$. Then there exists an algebra $H\in \mathfrak{X}$ containing $L$ such that its dimension is at most
$$
\max\{ \aleph_0, \dim L, |K|\}.
$$
Further, for any system $S$ of equations and in-equations over $L$,
we have either of the following assertions:

1- $S$ has a solution in $H$.

2- For any extension $H\subseteq E\in \mathfrak{X}$, the system $S$ has no solution in $E$.
\end{lemma}

\begin{proof}
We assume that $X$ is a countable set of variables and
$$
\eta=\max \{\aleph_0, \dim L, |K|\}.
$$
Any equation over $L$ consists of finitely many elements of $L$ and
$X$, so $\kappa$, the number of systems of equations and
in-equations over $L$, is  $|L\cup X|$. Note that $|L|=\max \{ \dim
L, |K|\}$, hence $\kappa=|L|+\aleph_0=\eta$.

We well-order the set of all  systems as $\{ S_{\alpha}\}_{\alpha}$, using ordinals $0\leq \alpha\leq \kappa$. Suppose $L_0=L$.
Let for any $0\leq \gamma\leq \alpha$, the algebra $L_{\gamma}\in \mathfrak{X}$ is defined in such a way that $|L_{\gamma}|\leq \kappa$ and
$$
\beta\leq \gamma \Rightarrow L_{\beta}\subseteq L_{\gamma}.
$$
We put
$$
E_{\alpha}=\bigcup_{\gamma\leq \alpha}L_{\gamma},
$$
so $E_{\alpha}\in \mathfrak{X}$ and further
$$
|E_{\alpha}|\leq \alpha |L_{\gamma}|\leq \kappa^2=\kappa.
$$
Suppose $S_{\alpha}$ has not solution in any extension of $E_{\alpha}$. Then we set $L_{\alpha}=E_{\alpha}$. If there is an extension
$E_{\alpha}\subseteq E\in \mathfrak{X}$ such that $S_{\alpha}$ has a solution $(u_1, \ldots, u_n)$ in $E$ ($n$ is the number of indeterminate
in $S_{\alpha}$), then we set
$$
L_{\alpha}=\langle E_{\alpha}, u_1, \ldots, u_n\rangle \subseteq E.
$$
Since $\mathfrak{X}$ is closed under subalgebra, so $L_{\alpha}\in \mathfrak{X}$ and also
$$
|L_{\alpha}|=|E_{\alpha}|\leq \kappa.
$$
Now, we define
$$
H=\bigcup_{0\leq \alpha\leq \kappa}L_{\alpha}
$$
which is an element of $\mathfrak{X}$. We have $|H|\leq \kappa^2=\kappa$ and hence
$$
\max\{ \dim H, |K|\}\leq \max \{\aleph_0, \dim L, |K|\},
$$
therefore we have
$$
\dim H\leq \max \{\aleph_0, \dim L, |K|\}.
$$

\end{proof}

\begin{theorem}
Let $\mathfrak{X}$ be an inductive class of (not necessarily associative) algebras over a field $K$. Suppose $\mathfrak{X}$ is closed under
subalgebra and $L\in \mathfrak{X}$. Then there exists an algebra $L^{\ast}\in\mathfrak{X}$ with the following properties,

1- $L$ is a subalgebra of $L^{\ast}$.

2- $L^{\ast}$ is existentially closed in the class $\mathfrak{X}$.

3- $\dim L^{\ast}\leq \max \{\aleph_0, \dim L, |K|\}$.
\end{theorem}

\begin{proof}
Let $H^0=L$ and $H^1=H$ be an algebra satisfying requirements of the previous lemma. Suppose $H^m$ is defined and let $H^{m+1}$ be an algebra
obtained by the lemma from $H^m$. We have
\begin{eqnarray*}
\dim H^{m+1}&\leq& \max \{ \aleph_0, \dim H^m, |K|\}\\
            &=&\max \{\aleph_0, \dim L, |K|\}.
\end{eqnarray*}
Now, suppose
$$
L^{\ast}=\bigcup_mH^m.
$$
This is an algebra having the properties 1-2-3.
\end{proof}

\section{Embedding of Lie algebras}
In \cite{shirvani} and \cite{Alon}, the concept of the HNN-extension
is defined for Lie algebras. Suppose $L$ is a Lie algebra over a
filed $K$ and $A$ is a subalgebra. Let $\delta:A\to L$ be a
derivation. Define a Lie algebra $L_{\delta}$ with the presentation
$$
L_{\delta}=\langle L, t:\ [t, a]=\delta(a); (a\in A)\rangle.
$$
The properties of this HNN-extension is studied in \cite{shirvani}
and \cite{Alon}. It is proved that $L$ is a subalgebra of
$L_{\delta}$. Similar constructions are also introduced for Lie
$p$-algebras and rings in \cite{shirvani}. In this section, using
this HNN-extension and the notion of existentially closed Lie
algebras, we obtain a new embedding theorem.

\begin{theorem}
Let $L$ be a Lie algebra over a field $K$. Then there exists a Lie algebra $L^{\ast}$ having the following properties,

1- $L$ is a subalgebra of $L^{\ast}$,

2- for any non-zero $a, b\in L^{\ast}$, there exists $x\in L^{\ast}$
such that $[x,a]=b$, and so $L^{\ast}$ is simple.

3- $\dim L^{\ast}\leq \max \{ \aleph_0, \dim L, |K|\}$,

4- $L^{\ast}$ is not finitely generated,

5- every finite dimensional simple Lie algebra over $K$ embeds in
$L^{\ast}$,

6- every finitely presented Lie algebra over $K$ embeds residually
in $L^{\ast}$,

7- if $K$ is finite, then every finite dimensional Lie algebra over
$K$ embeds in $L^{\ast}$,

8- if $K$ is finite and $A$ is  finite dimensional Lie algebra over
$K$, then we have
$$
Der(A)\cong \frac{N_{L^{\ast}}(A)}{C_{L^{\ast}}(A)}.
$$
\end{theorem}

\begin{proof}
We suppose that $\mathfrak{X}$ is the class of all Lie algebras and
then we apply the theorem 3.2. Hence, there exists an existentially
closed Lie algebra $L^{\ast}$ containing $L$ such that
$$
\dim L^{\ast}\leq \max \{ \aleph_0, \dim L, |K|\}.
$$
Let $0\neq a, b\in L^{\ast}$ and $\delta:\langle a\rangle \to L$ be the derivation $\delta(a)=b$. Consider the HNN-extension
$$
L^{\ast}_{\delta}=\langle L^{\ast}, t: [t, a]=b\rangle.
$$
We know that $L^{\ast}\leq L^{\ast}_{\delta}$ and the equation
$[x,a]=b$ has a solution in $L^{\ast}_{\delta}$, so 2 is proved.
This implies also that $L^{\ast}$ is simple.  To prove 4, suppose
$x_1, \ldots, x_n$ is a finite set of elements of $L^{\ast}$.
Consider the system
$$
[x, x_i]=0; (1\leq i\leq n), x\neq 0.
$$
This system has a solution in the Lie algebra $L^{\ast}\times \langle x\rangle$, and so we have $C_{L^{\ast}}(\langle x_1, \ldots, x_n\rangle)\neq 0$,
while $Z(L^{\ast})=0$. Therefore $L^{\ast}$ is not finitely generated.

Suppose $E$ is a finite dimensional simple Lie algebra with a basis $u_1, \ldots, u_n$ and suppose $[u_i, u_j]=\sum_r \lambda_{ij}^r u_r$.
Consider the system
$$
[x_i, x_j]=\sum_r\lambda_{ij}^r x_r; (1\leq i, j\leq n), x_i\neq 0; (1\leq i\leq n).
$$
This system has a solution in $L^{\ast}\times E$ and so there is a
non-zero homomorphism $E\to L^{\ast}$. Therefore $E$ embeds in
$L^{\ast}$. The proof of 6 is similar, so we prove 7. Suppose $K$ is
finite and let $E$ be a finite dimensional Lie algebra with a basis
$u_1, \ldots, u_n$ and structural constants $\lambda_{ij}^r$, i.e.
$[u_i,u_j]=\sum_r\lambda_{ij}^r u_r$. Let $T$ be the set of all
$n$-tuples $(a_1, \ldots, a_n)\in K^n$ with $a_i\neq 0$ for some
$i$. Consider the system
$$
[x_i, x_j]=\sum_r\lambda_{ij}^r x_r; (1\leq i,j\leq n)
$$
$$
\sum_r a_rx_r\neq 0; ( (a_1, \ldots,a_n)\in T).
$$
This system has a solution in $L^{\ast}\times E$, and so $E$ embeds in $L^{\ast}$.

Finally, to prove 8, let $K$ be finite and $A$ be finite dimensional
(subalgebra of $L^{\ast}$). Let $\delta\in Der(A)$ and consider the
HNN-extension
$$
L^{\ast}_{\delta}=\langle L^{\ast}, t: [t,a]=\delta(a); (a\in A)\rangle
$$
in which the system
$$
[x, a]=\delta(a); (a\in A)
$$
has a solution. So, there is an $x\in L^{\ast}$ such that $\delta(a)=[x,a]$ for all $a\in A$. Clearly $x\in N_{L^{\ast}}(A)$.
So there is an epimorphism $N_{L^{\ast}}(A)\to Der(A)$ with the kernel $C_{L^{\ast}}(A)$. So we have
$$
Der(A)\cong \frac{N_{L^{\ast}}(A)}{C_{L^{\ast}}(A)}.
$$

\end{proof}

As an interpretation of the above theorem, define a Lie algebra to
be {\em division Lie algebra}, if for all $a$ and $b$ with $a\neq
0$, there exists an $x$ such that $[a, x]=b$. Then the above theorem
shows that every Lie algebra can be embed in a simple division  Lie
algebra. However, that division Lie algebra is not finitely
generated. It can be asked that if there exists a finitely generated
division Lie algebra.

In \cite{shah}, we proved that if a Lie algebra $L$ over a filed of
characteristics zero has a finite dimensional ideal $U$ such that
$U^n\subseteq [A, U]$ for some abelian subalgebra $A$ and $n\geq 2$,
then $U$ is solvable. Here $[A, U]$ is the linear span of the set
$\{ [a, x]:\ a\in A, x\in U\}$. The Lie algebra $L^{\ast}$ which we
obtained in 5.1, shows that the assumption on the dimension of $U$
is essential, for example of we let $U=L^{\ast}$ and if we assume
that $A$ is any 1-dimensional subalgebra, then always $U^n\subseteq
[A, U]$, but clearly $U$ is not solvable ($L^{\ast}$ is simple).

\section{Existentially complete groups}
In this section, we define the notion of {\em existentially
complete} groups. It can be define the general frame of the
universal algebra. Existentially complete groups are special kind of
existentially closed groups, where in their definition, we use some
special sentences of the second order language of groups.

Let $G$ be a group and $A\leq G$ be a finitely generated subgroup.
Suppose $w(x_1, \ldots, x_n)\in G\ast F(x_1, \ldots, x_n)$. The
expression
$$
w(x_1, \ldots, x_n)\in A
$$
is called a {\em membership}. The negation of a membership is a {\em
non-membership}. We say that a finite system of memberships and
non-memberships (briefly, a membership system), is consistent, if a
there is a group $K$ containing $G$, such that the system has a
solution in $K$. A group $G$ with the property that every consistent
membership system has a solution in $G$, will be called
existentially complete. Clearly, such a group is in also
existentially closed. We can also define the notion of existentially
complete group in a given class. In some classes, the two notions
are equivalent, for example, in the class of locally finite groups,
there is no difference between two notions. We can define similarly,
existentially complete Lie algebras and also it is easy to see that
over finite fields, two properties of existentially closedness and
existentially completeness are equivalent.

A completely similar argument as in the section 1, the reader can
prove the next theorem. As a result, it shows that there exists
existentially complete groups.

\begin{theorem}
Let $\mathfrak{X}$ be an inductive class of groups, closed under
subgroup, such that $G\in \mathfrak{X}$. Then there exists an
existentially complete group $G^{\ast}$, containing $G$, such that
$$
|G^{\ast}|\leq \max\{ \aleph_0, |G|\}.
$$
\end{theorem}

The following embedding theorem, gives an interesting application of
the notion of existentially complete groups.

\begin{theorem}
Let $M$ be an existentially complete group and $A$ be a finitely
generated subgroup. Then, every extension of $A$ by  any finite
group embeds in $M$.
\end{theorem}

\begin{proof}
Let $A=\langle a_1, \ldots, a_m\rangle$ and $G$ be a group, having a
normal subgroup $B$, isomorphic to $A$. Let $\phi$ is the
isomorphism from $A$ to $B$ and $b_i=\phi(a_i)$. Suppose also that
$$
\frac{G}{B}=\{ g_1B, \ldots, g_nB\}.
$$
Let
$$
(g_iB)(g_jB)=g_{r(i,j)}B,
$$
where $r(i,j)$ is a function of $i$ and $j$. So, for any $i$ and
$j$, there exists an element $q_{ij}\in B$ such that
$$
g_{r(i,j)}^{-1}g_ig_j=q_{ij}.
$$
We also have $g_ib_jg_i^{-1}\in B$, so $g_ib_jg_i^{-1}=p_{ij}$, for
some $p_{ij}\in B$. Now, consider the next membership system
\begin{eqnarray}
x_{r(i,j)}^{-1}x_ix_j&=&\phi^{-1}(q_{ij}),\ \ \ \ \ \ \  (1\leq i,j\leq n)\\
         x_i^{-1}x_j&\not\in& A,\  \ \ \ \ \ \ \ \ \ \ \ \ \ \  (1\leq i<j\leq n)\\
         x_ia_jx_i^{-1}&=& \phi^{-1}(p_{ij}),\ \ \ \ \ \ \  (1\leq i\leq n, 1\leq
         j\leq m)
\end{eqnarray}
Let $M_1=\langle M\ast G: A=B, \phi\rangle$ be the amalgamated free
product of $M$ and $G$. We can verify easily that $(g_1, \ldots,
g_n)$ is a solution of the above system in $M_1$. Hence, the system
has a solution $(u_1, \ldots, u_n)$ in $M$. Let
$$
K=\langle a_1, \ldots, a_m, u_1, \ldots, u_n\rangle,
$$
which is a subgroup of $M$, containing $A$ as a normal subgroup, by
(3). This shows that every element of $K$ has the form $w(u_1,
\ldots, u_n)a$, such that $a\in A$ and $w$ is a group word in $u_1,
\ldots, u_n$. Now, by (1), we can replace every $u_iu_j$ by
$u_{r(i,j)}\phi^{-1}(q_{ij})$, and collecting all factors by this
method, we finally see that any element of $K$ has the unique  form
$u_ia$ for some $1\leq i\leq n$ and $a\in A$. Note that the same
observation is also true for the elements of $G$. Now, we prove that
$K\cong G$. Let $\psi:K\to G$ be define as $\psi(u_ia)=g_i\phi(a)$.
Clearly, $\psi$ is a bijection, so we show that it is a
homomorphism.

Let $k, k^{\prime}\in K$. We have $k=u_ia$ and
$k^{\prime}=u_ja^{\prime}$, for some $1\leq i, j\leq n$ and $a,
a^{\prime}\in A$. We have
$$
\psi(kk^{\prime})=\psi(u_iu_j(u_j^{-1}au_j)a^{\prime}).
$$
Note that by (3), we can write $u_j^{-1}au_j$ as $\phi^{-1}(p)$, for
some suitable $p\in B$. Also by (1), we have
$u_iu_j=u_{r(i,j)}\phi^{-1}(q_{ij})$, hence
\begin{eqnarray*}
\psi(kk^{\prime})&=&\psi(u_{r(i,j)}\phi^{-1}(q_{ij})\phi^{-1}(p)a^{\prime})\\
                 &=&g_{r(i,j)}q_{ij}p\phi(a^{\prime}).
\end{eqnarray*}
On the other hand, we have
\begin{eqnarray*}
\psi(k)\psi(k^{\prime})&=&g_i\phi(a)g_j\phi(a^{\prime})\\
                       &=&g_ig_j(g_j^{-1}\phi(a)g_j)\phi(a^{\prime})\\
                       &=&g_ig_jp\phi(a^{\prime})\\
                       &=&g_{r(i,j)}q_{ij}p\phi(a^{\prime}).
\end{eqnarray*}
This completes the proof.
\end{proof}

As a final note, we should say that the same result also can be
proved for Lie algebras.

{\bf Acknowledgement} The author would like to thank M. Kuzucuoglu
for comments and suggestions.

\end{document}